\documentclass[letterpaper,12pt,reqno,twoside]{amsart}
\usepackage[hmargin=3cm,vmargin=3cm]{geometry}

\usepackage{amssymb}
\usepackage{amsmath}
\usepackage[colorlinks,citecolor=blue,backref=page]{hyperref}
\usepackage[msc-links]{amsrefs} 

\newcommand{\be}{\begin{eqnarray}}
\newcommand{\ee}{\end{eqnarray}}

\newcommand{\beq}{\begin{equation}}
\newcommand{\eeq}{\end{equation}}
\newcommand{\beqn}{\begin{equation*}}
\newcommand{\eeqn}{\end{equation*}}

\newcommand{\defas}{\mathrel{\raise.095ex\hbox{$:$}\mkern-4.2mu=}}
\newcommand{\defasr}{\mathrel{=\mkern-4.2mu\raise.095ex\hbox{$:$}}}

\newcommand{\ie}{\textit{i.e.}}
\newcommand{\eg}{\textit{e.g.}}

\DeclareMathAlphabet{\mathfat}{U}{bbold}{m}{n}          

\DeclareMathOperator{\dist}{dist}
\DeclareMathOperator{\supp}{supp}

\newtheorem{thm}{Theorem}
\newtheorem{prop}[thm]{Proposition}

\newtheorem{lem}[thm]{Lemma}

\newtheorem{remark}[thm]{Remark}

\newcommand\cA{{\mathcal A}}
\newcommand\cB{{\mathcal B}}

\newcommand\bN{{\mathbb N}}

\newcommand\bR{{\mathbb R}}
\newcommand\bS{{\mathbb S}}

\newcommand\bZ{{\mathbb Z}}

\newcommand{\m}{{\mathrm m}}

\newcommand{\ve}{\varepsilon}

\begin{document}

\title[Positive Lyapunov exponent by a random perturbation]{Positive Lyapunov exponent by a random perturbation}

\author{Zeng Lian}
\author{Mikko Stenlund}

\address[Zeng Lian]{
Courant Institute of Mathematical Sciences\\
New York, NY 10012, USA.}
\email{lian@cims.nyu.edu}

\address[Mikko Stenlund]{
Courant Institute of Mathematical Sciences\\
New York, NY 10012, USA; Department of Mathematics and Statistics, P.O. Box 68, Fin-00014 University of Helsinki, Finland.}
\email{mikko@cims.nyu.edu}
\urladdr{http://www.math.helsinki.fi/mathphys/mikko.html}


\date{\today}

\begin{abstract}
We study the effect of a random perturbation on a one-parameter family of dynamical systems whose behavior in the absence of perturbation is ill understood. We provide conditions under which the perturbed system is ergodic and admits a positive Lyapunov exponent, with an explicit lower bound, for a large and controlled set of parameter values. 
\end{abstract}

\maketitle


\subsection*{Acknowledgements}
The authors are indebted to Lai-Sang Young for stimulating discussions. Mikko Stenlund has received funding from the Academy of Finland.


\section{Introduction}
\subsection{Background}
The question of the existence of positive Lyapunov exponents for a given dynamical system is one of extreme importance. It turns out to be very hard even in seemingly simple examples, which poses a great challenge to modern mathematics.

For example, consider the logistic family \footnote{Alternatively, as is done is some of the cited references, one could consider the real quadratic family $Q_a:x\mapsto a+x^2$, $a\in \bigl[-2,\frac14\bigr]$.}
\beqn
P_a:x\mapsto ax(1-x), \qquad a\in [1,4].
\eeqn
Let the set $\cA$ consist of those values of $a\in[1,4]$ for which $P_a$ admits a unique, finite, ergodic, absolutely continuous invariant measure, with a positive Lyapunov exponent. On the other hand, let $\cB$ denote the set of $a\in[1,4]$ for which $P_a$ has a periodic orbit attracting all orbits in $[0,1]$.
The set $\cA$ is known to have  positive measure~\cite{Jakobson1981} (also~\cite{BenedicksCarleson1985,Rychlik1988}) and the set $\cB$ is open and dense~\cite{GraczykSwiatek1997,Lyubich1997}. Using renormalization arguments, it has moreover been shown in~\cite{Lyubich2002} (see also~\cite{Lyubich1998}) that almost every value of~$a\in[1,4]$ falls in precisely one of the two sets $\cA$ and $\cB$. 

The above results on the abundance of parameter values admitting either a positive Lyapunov exponent or a periodic sink have been extended to multimodal situations in which several critical points are allowed; see~\cite{Tsujii1993,WangYoung2006} and~\cite{Kozlovski_Shen_vanStrien_2007}, respectively. However, it is not known whether the union of the two classes forms a set of full measure.

Taking these considerations into account, we understand why even such one-dimensional systems are notoriously hard to analyze: the nature of the dynamics depends very sensitively on the value of the parameter. Adding random noise to the model simplifies the picture due to averaging effects. For noisy systems, the dependence of a Lyapunov exponent on parameters is regular under mild conditions. Second,  it is a well-known dichotomy in the random case (see, \eg,~\cite{Young2008} for a discussion) that the sign of the Lyapunov exponent indicates in which dynamical category the system belongs to: a negative Lyapunov exponent implies convergence to a random sink consisting of finitely many points for almost all sample paths~\cite{LeJan1985, Baxendale1992}, while a positive one yields a random SRB measure almost surely~\cite{LedrappierYoung1988}.

Regarding random perturbations of dynamical systems, it is commonplace to start out with systems that are very well controlled in the absence of perturbation. One then goes on to show that control of the system is retained under sufficiently small random perturbations. A system possessing this property is called stochastically stable. For uniformly hyperbolic systems, standard references on stochastic stability include~\cite{Kifer1988} and~\cite{Young1986}. For one-dimensional maps admitting critical points, see~\cite{KatokKifer1986} and~\cite{BenedicksYoung1992}. The more recent~\cite{BenedicksViana2006} discusses a two-dimensional case. The preceding list, which of course could be continued much further, is meant to point the reader quickly to a handful of original references. 

While results of the above kind are very interesting, it would be much more satisfying if one could reverse the direction. That is, to perturb a dynamical system too hard to analyze by itself, to take advantage of the randomness in the noisy system, and then to infer properties of the zero-noise limit. The idea of doing so can be traced back at least to Pontryagin, Andronov, and Vitt \cite{PontryaginAndronovVitt1933}, and later to Kolmogorov~\cite{Sinai1989}. As the real world is inherently noisy, say, an invariant measure obtained in that limit could be interpreted as a physically observable (albeit idealized) statistical description of the system. Unfortunately, the zero-noise limit is not always well behaved. For instance, Lyapunov exponents may fail to be continuous at the point of vanishing perturbation; see the figure-eight attractor in~\cite{CowiesonYoung2005}. 

To take steps in the direction of the previous paragraph --- and more generally to develop new techniques for proving lower bounds on Lyapunov exponents --- we work with a one-parameter family of systems in which the dynamical properties of the unperturbed system for a given parameter value are unknown. Here, a sufficiently large perturbation is required (i) to regularize the parameter dependence of the nature of the dynamics so that (ii) a good lower bound on the Lyapunov exponent can be established for a large and controlled set of parameters.

\vspace{5mm}
The paper is organized as follows. In Section~\ref{subsec:prelim} we introduce our model and the necessary technical notions so that the results of the paper can be formulated in Section~\ref{subsec:results}. Theorem~\ref{thm:erg} concerns ergodicity of the system and is proved in Section~\ref{sec:erg}. Proposition~\ref{prop:sinks} identifies parameter values for which random sinks appear unless the perturbation is large enough. Its proof is given in Section~\ref{sec:Lyap}. Theorems~\ref{thm:large_smear} and~\ref{thm:crit_smear} give sufficient conditions for a positive Lyapunov exponent together with an explicit lower bound. They are also proved in Section~\ref{sec:Lyap}.

\subsection{Preliminaries}\label{subsec:prelim}
We denote by $\bS$ the circle obtained by identifying the endpoints of the unit interval $[0,1]$ and by $\m$ the uniform measure on $\bS$.  Let $$\tau_a:\bS\to \bS:x\mapsto a+ x + L\psi(x) \pmod 1,$$ where $a\in[0,1)$ and $L>0$ are constants, and $\psi:\bS\to\bR$ is a twice continuously differentiable map. Although $\tau_a$ depends on $L$, it is notationally convenient not to indicate this explicitly by a subscript. We assume that $\psi$ has $N>0$ critical points $c_1,\dots,c_N$ where $\psi'(c_i)=0$, each of which is nondegenerate, \ie, $\psi''(c_i)\neq 0$. Since nondegenerate critical points are isolated and the circle is compact, $N<\infty$. 

Notice that the maps $\tau_a$ are not unimodal with just one critical point, which is a case studied extensively in the literature. By contrast, the present paper involves a rather general class of multimodal maps for which the number of critical points is arbitrary. This setting arises, for example, in applications pertaining to shear-induced chaos in the theory of rank one attractors~\cite{WangYoung2002,WangYoung2003,WangYoung2006, LinYoung2010,OttStenlund2010}.

As discussed earlier, the parametric dependence of the dynamical nature of such maps can be very complicated. With the exception of some special parameter values, it is practically impossible to determine whether a particular choice of the parameter~$a$ results in chaotic or regular motion. To remedy the situation, we add a small amount of random noise to the system, which turns out to have a regularizing effect on the dependence of the dynamics on~$a$.

To be specific, we are interested in the ergodic and chaotic properties of the random circle map $\tau_a+Y$ determined by
$$
x\mapsto \tau_a(x)+Y \pmod1,
$$
where $Y$ is a random perturbation, or kick, distributed uniformly on $[-\ve,\ve]$ with some $\ve>0$. The reader may think of $\ve$ as the level of noise present in the system. For simplicity, we take  $Y(\omega)=\omega$ for each realization $\omega\in[-\ve,\ve]$. Given a realized sequence $(\omega_n)_{n=1}^\infty\in [-\ve,\ve]^{\bZ_+}$ of i.i.d.\ kicks, the trajectory $(x_n)_{n=0}^\infty$ of any initial point $x_0=x\in\bS$ is determined for each $n\geq 1$ by 
$$
x_{n} = \tau_a(x_{n-1})+\omega_n.
$$

Let us write 
$$
\tau_a^\omega(x)=\tau_a(x)+\omega \qquad \forall\,\omega\in\bR
$$
and denote the uniform probability measure on $[-\ve,\ve]$ by $\eta$. We say that a Borel probability measure $\mu$ on $\bS$ is \emph{invariant} for the above random map $\tau_a+Y$, if
\beqn
\mu(B) = \int_{[-\ve,\ve]} \mu((\tau_a^\omega)^{-1}B) \,d\eta(\omega) 
\eeqn
holds for every Borel set $B\subset \bS$. Further, an invariant measure $\mu$ is \emph{ergodic}, if the condition $\mu(B\bigtriangleup(\tau_a^\omega)^{-1}B)=0$ for $\eta$-a.e.\ $\omega\in[-\ve,\ve]$ on the set $B$ implies $\mu(B)\in\{0,1\}$. Here $A\bigtriangleup B$ denotes the symmetric difference $(A\setminus B) \cup (B\setminus A)$ of two sets. For example, it is easy to check that for complete smearing ($\ve=\frac12$) of the image $\tau_a(x)$, the unique invariant measure --- which is always ergodic --- is $\mu=\m$. In general, there can exist many invariant measures. However, we will see in Lemma~\ref{lem:unique} that there can be at most one measure which is both ergodic and equivalent to $\m$. For a sufficiently noisy system, such a measure turns out to exist and to rule out the existence of other invariant measures.

It follows from Birkhoff's ergodic theorem that the \emph{Lyapunov exponent} 
\beq\label{eq:Lyap_def}
\lambda_a\bigl((\omega_n)_{n=1}^\infty,x;L\big) = \lim_{n\to\infty}\frac1n\sum_{k=0}^{n-1} \log|\tau_a'(x_k)|
\eeq
exists ($\eta^{\bZ_+}\times\mu$)-a.s., if $\mu$ is invariant for the random map $\tau_a+Y$. If $\mu$ is also ergodic, then $\lambda_a\big((\omega_n)_{n=1}^\infty,x;L\big)$ is ($\eta^{\bZ_+}\times\mu$)-a.s.\ equal to the constant
\beq\label{eq:Lyap}
\lambda_a(L) = \int_\bS\log|\tau_a'|\,d\mu.
\eeq

Notice that the Lyapunov exponent measures the exponential rate of separation of initial points infinitesimally close to $x$, under the same sequence $(\omega_n)_{n=1}^\infty$ of kicks.  A positive Lyapunov exponent indicates sensitive dependence of the trajectory on the initial condition.

For any $K>1$, the map $\tau_a$ is uniformly expanding on the complement of the set
$$I_K=\{x\in \bS\,:\, |\tau_a'(x)|\leq K\}.$$ The non-expanding set $I_1$ is critical to the dynamics and its structure plays a central role in our analysis.

For any $K\geq 1$, the set $I_{K}$ is independent of $a$ and consists of those points $x\in\bS$ for which $-(K+1)/L\leq \psi'(x)\leq (K-1)/L$.  For a sufficiently large $L$,  they form $N$ disjoint intervals, each containing precisely one of the critical points $c_i$ of $\psi$, which does not depend on  $L$. (For smaller values of $L$, some of the intervals merge.) The endpoints are obtained by solving $\psi'(c_i+\xi_i)=\mp(K\pm 1)/L$ for $\xi_i$. Taylor expanding, we see that the length of such an interval is ${2K/|\psi''(c_i)| L} + O((K/L)^2)$. In particular, the length of the largest component of $I_K$ is
\beq\label{eq:b_K}
b_K = \frac{2K}{L}\frac{1}{\min_{1\leq i\leq N} |\psi''(c_i)|} + O\bigl((K/L)^2\bigr)\quad \text{as}\quad  L/K\to\infty,
\eeq
and
\beq\label{eq:I_K}
\m(I_K) = \frac{2K}{L}\sum_{1\leq i\leq N}\frac{1}{|\psi''(c_i)|} + O\bigl((K/L)^2\bigr)\quad \text{as}\quad  L/K\to\infty.
\eeq

Throughout this paper, 
$$B_r(A)=\{x\in\bS\,:\,\dist(x,A)\le r\}$$
is the closed $r$-neighborhood of a set $A\subset \bS$ and $B_r(x)=B_r(\{x\})$. In our estimates, $C$ stands for a generic constant whose numerical value may change from one expression to the next.

\subsection{Results}\label{subsec:results} 
We are now in position to state the results of this paper.
\begin{thm}\label{thm:erg}
Given any $L>0$, $a\in[0,1)$, and $\ve>\m(I_{N+1})/2$, the system admits a unique invariant measure. It is both ergodic and equivalent to $\m$. For any sufficiently large value of $L$, the same is true for any $a\in[0,1)$ and $\ve>b_2(L)/2$.
\end{thm}

\begin{thm}\label{thm:large_smear}
Given constants $C>0$ and $\beta\in(0,1]$, a function $\ve=\ve(L)\ge C L^{\beta-1}$, and a sufficiently large $L$, the system admits a unique ergodic measure for any value of $a$. With respect to those measures,
$$\liminf_{L\to\infty}\inf_{a\in [0,1)} \frac{\lambda_a(L)}{\log L}\ge \beta.$$
\end{thm}
\begin{remark}\label{rem:erg_implicit}
The assumption on the size of the perturbation $\ve$ guarantees ergodicity. Indeed, recalling~\eqref{eq:b_K}, the condition $\ve>b_2(L)/2$ of Theorem~\ref{thm:erg} is satisfied if $L$ is large.
\end{remark}
Let us pause to discuss the cases $\beta=1$ and $\beta=0$. Assuming that the measure $\mu$ is ergodic, with no explicit constraints on the parameters, Jensen's inequality implies, for all $L\geq 1$,
\beqn
\exp \lambda_a(L) = \exp\int_{\bS} \log|\tau_a'|\,d\mu \leq \int_{\bS} 1+L|\psi'|\,d\mu \leq CL,
\eeqn
or 
\beq\label{eq:upper_bound}
\limsup_{L\to\infty}\sup_{a\in [0,1)}\frac{\lambda_a(L)}{\log L} \leq 1.
\eeq
Hence, the lower bound of Theorem~\ref{thm:erg} for $\beta=1$ is optimal. 

When $\beta=0$, the theorem above suggests that all values of $a$ may not yield a positive limit for $\lambda_a(L)/\log L$. The reason is that the size of the non-expanding set $I_{1}$ scales like $L^{-1}$, as observed above. If the perturbation is not sufficiently large in comparison, the smeared image $B_\ve(\tau_a(z))$ of a critical point $z$ of $\tau_a$ may, for some values of the parameter~$a$, be contained in a small contracting neighborhood of~$z$. By this mechanism, negative Lyapunov exponents appear:  

\begin{prop}\label{prop:sinks}
There exists a constant $D>0$ and, for any large enough $L$ and for any $\ve\leq DL^{-1}$, an invariant measure $\mu$, such that
\beqn
\lambda_a\bigl((\omega_n)_{n=1}^\infty,x;L\big)<0
\eeqn
for ($\eta^{\bZ_+}\times\mu$)-a.e.\ $((\omega_n)_{n=1}^\infty,x)$.
\end{prop}
We point out that, by general results~\cite{LeJan1985, Baxendale1992},  a negative Lyapunov exponent for an ergodic random system implies the existence of a random attracting set consisting of finitely many points. Such a set can in fact be  constructed following the proof of the above proposition.

To shed more light on the case $\beta=0$ in particular, we have the following result, which holds for a \emph{restricted set of values} of the parameter $a$.

\begin{thm}\label{thm:crit_smear}
There exist sets $A_L\subset[0,1)$, $L>0$, such that the following holds. Each $A_L$ consists of a union of at most $N^2$ intervals and  
$$\lim_{L\to\infty}\m(A_L)=1.$$
Given a function $\ve=\ve(L) > b_2(L)/2$
the system admits a unique ergodic measure for any sufficiently large value of $L$ and any value of $a$.
With respect to those measures, 
 $$
\liminf_{L\to\infty} \inf_{a\in A_L}\frac{\lambda_a(L)}{\log L}\ge \frac12.
$$
\end{thm}

The number $\frac 12$ appearing in Theorem~\ref{thm:crit_smear}  is a technical artifact of the proof. Namely, the set $A_L$ consists roughly speaking of parameter values for which the set $I_1$ does not intersect its image $B_\ve(\tau_a(I_1))$ under the random map $\tau_a+Y$. Clearly one could exclude fewer parameter values by considering cases in which the trajectory of a point is allowed to visit the set $I_{1}$ at several consequtive times. Moreover, the estimate could be improved for many parameter values by more elaborate techniques. In view of the fact that the value $1$ could not be exceeded due to the upper bound in \eqref{eq:upper_bound}, we have not pursued such an improvement.


\section{Ergodicity}\label{sec:erg}

We will next prove Theorem~\ref{thm:erg}. But first we need to recall some basic facts and definitions.

Notice that we can view the random circle map $\tau_a+Y$ above as the Markov chain generated by the
transition kernel
\beq\label{eq:kernel}
p(x,A) = \frac{1}{2\ve}\m(A\cap B_\ve(\tau_a(x))),
\eeq
for points $x\in\bS$ and Borel sets $A\subset\bS$. In other words, $p(x,A)$ is the probability that the random image of the initial point $x$ belongs to the set $A$. 
A Borel measure $\mu$ on $\bS$ is stationary if 
\beq\label{eq:stationary}
\mu(A)=\int_{\bS}p(x,A)\,d\mu(x).
\eeq
Stationary measures for the Markov chain are precisely the invariant measures for the random map.
Because $x\mapsto p(x,A)$ is a continuous function for any $A$, it is a standard fact~\cite{Kifer1986} that there exists a stationary measure $\mu$. For any $\ve>0$,
\beq\label{eq:abs_cont}
\mu(A) \leq \max_x p(x,A) \leq \frac{1}{2\ve} \m(A)
\eeq
follows immediately from \eqref{eq:stationary}. In other words, $\mu$ is absolutely continuous with respect to the measure $\m$, written $\mu\ll \m$, and therefore has a density~$\rho$:
\beqn
\mu(A) = \int_A \rho\,d\m.
\eeqn

Define $P^*:L^1(\bS)\to L^1(\bS)$ by
$$(P^*f)(x)=\int_{\bS}f(y)\,p(x,dy),\quad f\in L^1(\bS).$$
Thus, given an initial state $x\in\bS$ of the system, the expected value of a function $f\in L^1(\bS)$ after one time step is $(P^*f)(x)$. 
We say that a Borel set $A$ is invariant modulo $\mu$ if $(P^*1_A)(x)=1_A(x)$  for $\mu$-a.e.~$x$. Notice that in this case $p(x,A)=1$  for $\mu$-a.e.~$x\in A$. Finally, the measure $\mu$ is \emph{ergodic}, if all invariant sets are trivial, \ie, $\mu(A)\in\{0,1\}$ whenever $A$ is invariant modulo $\mu$. Again, this definition of ergodicity coincides with the one given earlier for the random map.

\begin{lem}\label{lem:unique}
Let $\ve>0$ and consider the Markov chain generated by the transition kernel $p$. If there exists an ergodic stationary measure which is equivalent to $\m$, there are no other stationary measures.
\end{lem}
\begin{proof} 
Let $\mu_1$ and $\mu_2$ be ergodic stationary measures for the Markov chain on the state space $\bS$. Recall that, given an initial measure $\mu$, the chain generates a probability measure $P^\mu$ on the space of trajectories, $\bS^\bN$, and
\beq\label{eq:path_decomp}
P^\mu = \int_\bS P^{x}\,d\mu(x),
\eeq
where $P^x$ is the measure corresponding to an initial point mass at $x$. In the sense of measure preserving transformations, $P^{\mu_1}$ and $P^{\mu_2}$ are ergodic with respect to the left shift on $\bS^\bN$. It follows from Birkhoff's ergodic theorem that either $P^{\mu_1}=P^{\mu_2}$ or the measures are mutually singular, written $P^{\mu_1} \perp P^{\mu_2}$. In the first case, $\mu_1=\mu_2$, as can be seen by considering sets of the form $A\times\bS^{\bZ_+}\subset \bS^\bN$ with $A\subset \bS$ a Borel set. In the second case, there exists a Borel set $\cA\subset\bS^\bN$ such that $P^{\mu_1}(\cA)=1$ and $P^{\mu_2}(\cA)=0$. Therefore \eqref{eq:path_decomp} implies $P^x(\cA)=1$  for $\mu_1$-a.e.~$x$ and $P^x(\cA)=0$ for $\mu_2$-a.e.~$x$, meaning that $\mu_1 \perp \mu_2$. 

In conclusion, two distinct ergodic measures are mutually singular. Assuming now that there exists an ergodic measure $\mu$ which is equivalent to $\m$, it must be the only ergodic measure, because by \eqref{eq:abs_cont} any other candidate would also have a density with respect to $\m$. Since any stationary measure is a convex combination of ergodic ones, $\mu$~must in fact be the only stationary measure for the Markov chain. 
\end{proof}

\begin{proof}[Proof of Theorem~\ref{thm:erg}]
Uniqueness of the measure $\mu$ with the claimed properties is guaranteed by Lemma~\ref{lem:unique}. Thus, we are left with proving existence. 

First, we claim that for $\m$-a.e.\ $x\in \supp \mu$, it holds true that $p(x,\supp \mu)=1$. Since $d\mu=\rho\,d\m$ is stationary
\begin{align*}
 \mu(\supp \mu)=\int_\bS p(x,\supp \mu)\,d\mu(x)
=\int_{\supp \mu}p(x,\supp \mu)\rho(x) \,d\m(x).
\end{align*}
The claim follows from $\mu(\supp \mu)=\int_{\supp \mu}\rho(x)\,d\m(x)$. Moreover, since $\mu\ll\m$, we also have $\m(\supp \mu)>0$. 

Second, we show that any Borel set $A$ invariant modulo $\mu$ with $\mu(A)>0$ has $\m(A)=1$. By  $\mu\ll\m$, this also implies $\mu(A)=1$. Since $\supp\mu$ is invariant modulo $\mu$, we assume without loss of generality that $A\subset\supp\mu$. The idea of the proof is to construct a sequence of intervals $J_0,J_1,\dots$ with $J_{i+1}=B_\ve(\tau_a (J_i))$ such that {\bf (i)}~$J_i\subset A \mod\m$, meaning $\m(J_i\setminus A)=0$ and {\bf (ii)}~$J_i=\bS$ for all sufficiently large values of~$i$, provided $\ve$ is sufficiently large. As a byproduct, we will have obtained $\m(\supp\mu)=1$, which shows that $\m$ and $\mu$ are equivalent measures.

Proof of {\bf (i)}.
Note that restricted to $\supp \mu\supset A$, the statements ``$\mu$-a.e.'' and ``$\m$-a.e.'' are equivalent, so we will simply write ``a.e.'' in such a situation.
We fix an arbitrary parameter value $a\in[0,1)$ for the map $\tau_a$. For a.e.\ $x\in A$, 
$p(x,A)=1.$ We pick such an $x$. Then, by \eqref{eq:kernel}, the interval $J_0=B_\ve(\tau_a(x))$ satisfies $J_0\subset A\mod \m$. Observe that $\mu(J_0\cap A)>0$ because $\m(J_0\cap A)=2\ve$ and $A\subset \supp \mu$.  By invariance of $A$, $p(y,A)=1$
a.e.\ $y\in J_0\cap A$. Denote the set of such
$y$ by $\tilde J_0$. Then $J_0=\tilde J_0\cup N_0$ for some $\m$-null set $N_0$.

We define $J_{i}=B_\ve(\tau_a(J_{i-1}))$
 and $\tilde J_{i}=\{y\in J_i\cap A\,:\, p(y,A)=1\}\subset A$ for $i\geq 1$ inductively. We claim that
$J_i=\tilde J_i \cup N_i$ for some $\m$-null set $N_i$. The proof is inductive. First of all, denoting by $\partial J_i$ the boundary of $J_i$ (consisting of no more than two points), we have
\beqn\label{eq:J_i}
J_i = \bigcup_{y\in \hat J_{i-1}} \!\! B_\ve(\tau_a(y)) \cup\partial J_i,
\eeqn
where $\hat J_{i-1}$ is a countable dense subset of $\tilde J_{i-1}$. This is so, because $J_i$ and $J_{i-1}=\tilde J_{i-1}\cup N_{i-1}$ are closed intervals and $\m(N_{i-1})=0$. Second, for each $y\in\tilde J_{i-1}$
 we have $p(y,A)=1$, which implies 
$B_\ve(\tau_a(y))\subset A\mod\m$. Since $\hat J_{i-1}$ is countable, we conclude $J_i\subset A\mod \m$. Hence also $\tilde J_i=J_i\mod\m$ by invariance of $A$. 

Proof of {\bf (ii)}. Recall that $N$ is the number of critical points of the map $\psi$ and that $I_K = \{x\in\bS\,:\,|\tau_a'(x)|\leq K\}$. Suppose first that $L>0$ is arbitrary. For any $K>N+1$, any $\ve\geq K\m(I_K)/2(N+1)$, and any interval $J\subset \bS$, we have that
\begin{align*}
\m(B_{\ve}(\tau_a(J))) & \geq \min\{1,2\ve+\m(\tau_a(J\cap (I_{K})^c))\}
\\
&
\ge\min\left\{1,2\ve+\frac{K}{N+1}\m(J\cap (I_{K})^c)\right\}
\\
&\ge\min\left\{1,2\ve+\frac{K}{N+1}(\m(J)-\m(I_{K}))\right\}
\\
&\ge\min\left\{1,\frac{K}{N+1} \m(J)\right\}.
\end{align*}
The maximum number $N$ of critical points enters the argument, because although the map $\tau_a$ is locally expanding on $(I_K)^c$, the graph of $\tau_a$ has a fold at each critical point.  The above estimate shows that the interval $J_i$ grows exponentially with $i$ until it covers $\bS$.

Now, assume instead that $L>0$ is so large that $\tau_a$ wraps each of the intervals $(z_i,z_{i+1})$ around $\bS$ twice, where $z_i$ are the critical points of $\tau_a$ labeled clockwise, and that each $z_i$ is nondegenerate. Let the interval $J\subset \bS$ be such that $B_\ve(\tau_a(J))\neq\bS$ --- otherwise we are done. Then $J$ contains at most one of the critical points $z_i$ and intersects at most one component of $I_K$. Recall that $b_{K}$ denotes the length of the largest component of $I_K$. For any $K>2$ and any $\ve\geq K b_K/4$,
\begin{align*}
\m(B_{\ve}(\tau_a(J))) & \geq 2\ve+\m(\tau_a(J\cap (I_{K})^c))
\\
&
\ge 2\ve+\frac{K}{2}\m(J\cap (I_{K})^c)
\\
&\ge 2\ve+\frac{K}{2}(\m(J)-b_{K})
\\
&\ge \frac{K}{2} \m(J).
\end{align*}
Again, we are able to conclude that the interval $J_i$ grows exponentially with $i$ until it covers $\bS$.

As a final remark,  the value of $K$ was arbitrary. We see that, for arbitrary $L$ and for large $L$,  it is enough to assume $\ve>\m(I_{N+1})/2$ and $\ve>b_2/2$, respectively, for the theorem to hold. \end{proof}


\section{Lyapunov exponent}\label{sec:Lyap}
In this section we prove our main results, Theorems~\ref{thm:large_smear} and~\ref{thm:crit_smear}.  Before that, we present a short proof of Proposition~\ref{prop:sinks} on the existence of negative Lyapunov exponents for moderate size perturbations.
\begin{proof}[Proof of Proposition~\ref{prop:sinks}] 
Fix $i\in\{1,\dots,N\}$. If $L$ is sufficiently large, the map $\tau_a$ has a critical point $z\in\bS$ close to $c_i$. Now, tune $a_z\in \bS$ so that $\tau_{a_z}(z)=z$. Taylor expanding at $z$,
\beqn
\tau_a^\omega(x) = \omega+a-a_z + \tau_{a_z}(x) = \omega + a - a_z + z + L\int_z^x(x-t)\psi''(t)\,dt.
\eeqn
We see that 
$$
\tau_a^\omega(B_\nu(z))\subset B_\nu(z)
$$
for any $\omega\in[-\ve,\ve]$, provided that
$$
\ve + |a-a_z| + \frac L 2 \sup|\psi''|\,\nu^2 \leq \nu.
$$
Moreover,
\beq
\label{eq:difference}
|\tau_a(x)-\tau_a(y)| \leq ML\nu\, |x-y| \qquad \forall\,x,y\in B_\nu(z),
\eeq
where the constant $M$ only depends on $\psi$. Let us now choose $\nu=\frac 1{2ML}$ and $\ve= \frac \nu3$.  For any $a$ with $|a-a_z|\leq \frac \nu3$ and a large enough $L$, any realization of the random map $\tau_a+Y$ maps the interval $B_\nu(z)$ inside itself. By the same argument as in the beginning of Section~\ref{sec:erg}, the map $\tau_a|_{B_\nu(z)}+Y$ has an invariant measure $\mu$. This is an invariant measure for $\tau_a+Y$ supported on a subset of $B_\nu(z)$. Since \eqref{eq:difference} implies that $|\tau_a'|\leq \frac12$ on $B_\nu(z)$, we obtain directly from \eqref{eq:Lyap_def} the bound
\beqn
\lambda_a\bigl((\omega_n)_{n=1}^\infty,x;L\big)\leq -\log 2
\eeqn
for ($\eta^{\bZ_+}\times\mu$)-a.e.\ $((\omega_n)_{n=1}^\infty,x)$.
\end{proof}

To estimate the Lyapunov exponent from below, we first need to bound the invariant density $\rho$ from above.

Notice that $p(x,\cdot)$ in \eqref{eq:kernel} is a Borel probability measure and that it has the representation
\beqn
p(x,A) = \int_A \phi(x,y)\,d\m(y),
\eeqn
where the density $\phi(x,\cdot)$ is the Radon--Nikodym derivative
\beqn
\phi(x,y) = \frac{d p(x,\cdot)}{d\m}\bigg|_y = \frac{1}{2\ve} 1_{B_\ve(\tau_a(x))}(y).
\eeqn
Iterating \eqref{eq:stationary} once,
\beq\label{eq:kernel2}
\mu(A) = \int_\bS p(x,A)\,d\mu(x) = \iint_{\bS\times\bS} p(x,dy)\,p(y,A)\,d\mu(x) .
\eeq
Recall that $\mu$ is absolutely continuous with density $\rho$. Thus,
\beqn
\lim_{\delta\to 0+}\frac{\mu(B_\delta(x_0))}{\m(B_\delta(x_0))} = \rho(x_0).
\eeqn
Applying the bounded convergence theorem to \eqref{eq:kernel2} with $A=B_\delta(x_0)$,
\beqn
\begin{split}
\rho(x_0) &= \int_\bS  \phi(x,x_0)\, d\mu(x) = \iint_{\bS\times\bS}  \phi(x,y)\phi(y,x_0) \,d\m(y)\,d\mu(x).
\end{split}
\eeqn
The first equality immediately yields the bound
\beq\label{eq:bound}
\rho(x_0) \leq \frac{1}{2\ve},
\eeq
whereas the second one shows that
\beqn
\rho(x_0) \leq \max_x\int_\bS \phi(x,y)\phi(y,x_0) \,d\m(y).
\eeqn
Here
\beqn
\begin{split}
\phi(x,y)\phi(y,x_0) &= \frac{1}{4\ve^2} 1_{B_\ve(\tau_a(x))}(y)\cdot1_{B_\ve(\tau_a(y))}(x_0)
\end{split}
\eeqn
so that 
\beq\label{eq:bound2}
\rho(x_0) \leq  \frac{1}{4\ve^2} \max_z \m\bigl(B_\ve(z)\, \cap \,\tau_a^{-1}B_\ve(x_0)\bigr).
\eeq

It turns out that we also need the following estimate.
\begin{lem}\label{lem:log_int}
There exists a constant $C>0$ such that, for sufficiently large values of $L>0$,
\beqn
\int_{I_1}\log|\tau_a'|\,d\m \geq - CL^{-1} .
\eeqn
\end{lem}
\begin{proof}
Recall that the critical points of $\psi$ are nondegenerate and observe that the set $I_1$ is precisely $\{x\in\bS\,:\,-2/L\leq \psi'(x)\leq 0\}$. Thus, for large $L$, $I_1$ consists of $N$ disjoint intervals, none of which contains any zeros of $\psi''$. Thus, $\inf_{I_1}|\psi''| > 0$. Moreover, $I_1$ is the union of $2N$ intervals $I_1^{(1)},\dots,I_1^{(2N)}$ on the interior of each of which $|\tau_a'|$ is one-to-one and onto $(0,1)$. Therefore, by the change of variables $t=|\tau_a'|$ and the fact that $\int_0^1\log t = -1$,
\begin{align*}
\int_{I_1}\log
|\tau_a'|\,d\m
& = \sum_{i=1}^{2N} \int_{I_1^{(i)}}\log
|\tau_a'|\,d\m \geq \sum_{i=1}^{2N} \frac{1}{\inf_{I_1^{i}}|\tau_a''|}\int_0^1\log t\,dt 
\\
& \ge -\frac{2N}{\inf_{I_1}|\tau_a''|}
 \ge -\frac{2N}{L \inf_{I_1}|\psi''|}.
\end{align*}
Since $\psi$ is independent of any parameters, the proof is complete.
\end{proof}

\begin{proof}[Proof of Theorem~\ref{thm:large_smear}]
We first consider the case $\beta\in(0,1)$, \ie, $\beta\neq 1$. 
Since $\ve \ge C L^{\beta-1}$, the bound \eqref{eq:bound} on the density $\rho$ of $\mu$, together with~\eqref{eq:I_K}, yields
\beqn
\rho\leq CL^{1-\beta}
\quad\text{and}\quad
\mu\bigl(I_{L^{\hat\beta}}\bigr)\le \sup\rho\cdot \m\bigl(I_{L^{\hat\beta}}\bigr) \le CL^{\hat\beta-\beta}
\eeqn
uniformly for $\hat\beta\in(0,\beta)$. 
The conditions of Theorem~\ref{thm:erg} are satisfied for large enough~$L$. We can therefore use the formula in~\eqref{eq:Lyap} for ergodic measures $\mu$ to bound the Lyapunov exponent $\lambda_a(L)$:
\begin{align*}
 \lambda_a(L)
&\ge \int_{(I_{L^{\hat\beta}})^c}\log |\tau_a'|\,d\mu+\int_{I_1}\log
|\tau_a'|\,d\mu\\
&\ge \Bigl(1-\mu\bigl(I_{L^{\hat\beta}}\bigr)\Bigr)\log L^{\hat\beta}+\sup_{I_1}\rho\cdot\int_{I_1}\log
|\tau_a'|\,d\m\\
&\ge (1-CL^{\hat\beta-\beta})\log L^{\hat\beta}-
C L^{1-\beta}\cdot CL^{-1}
\\
& \ge \Bigl(\bigl(1-o(1)\bigr)\hat\beta-o(1)\Bigr)\log L.
\end{align*}
Above, Lemma~\ref{lem:log_int} was used to bound the last integral.
Since $\hat \beta$ can be chosen arbitrarily close to $\beta$, 
the proof is complete for $\beta\neq 1$.

In order to analyze the case $\beta=1$, we replace $I_{L^{\hat\beta}}$ by $I_{h(L)}$, where $h(L)=L/\log L$. Notice that $\ve \ge C$ results in
\beqn
\rho\leq C\quad\text{and}\quad\mu(I_{h(L)})\leq C/\log L
\eeqn
by the same arguments as above. Therefore,
\begin{align*}
 \lambda_a(L)
&\ge \int_{(I_{h(L)})^c}\log |\tau_a'|\,d\mu+\int_{I_1}\log
|\tau_a'|\,d\mu\\
&\ge \Bigl(1-\mu\bigl(I_{h(L)}\bigr)\Bigr)\log (L/\log L)+\sup_{I_1}\rho\cdot\int_{I_1}\log
|\tau_a'|\,d\m\\
&\ge (1-C/\log L)(\log L-\log\log L)-
 CL^{-1}
\\
& \ge \bigl(1-o(1)\bigr)\log L,
\end{align*}
which proves the theorem also for $\beta=1$.
\end{proof}


\begin{proof}[Proof of Theorem~\ref{thm:crit_smear}]
Below, we will specify a set $A_L$, taking $a$ from which a lower bound on $\lambda_a(L)$ can be deduced.
 
It will be helpful to keep in mind that, for large enough $L$, there are precisely $N$ critical points of $\tau_a$, all nondegenerate, which are $O(L^{-1})$ units apart from the critical points $c_1,\dots,c_N$ of the map $\psi$.

Recall that $\ve$ depends on $L$. Let us first assume that there exists a non-increasing positive function $\ve_0(L)$, and point out the existence of a constant $C>0$, such that
\beqn
\text{$C^{-1}L^{-1} \leq \ve\leq \ve_0(L)$ for any sufficiently large $L$}
\quad \text{and}\quad
\lim_{L\to\infty}\ve_0(L) = 0.
\eeqn
For any pair $K_1,K_2\geq 1$ and any $\ve>0$ we define 
\beq\label{eq:A_L}
A^{K_1,K_2}_{L,\epsilon}=\{a\in[0,1)\,:\, B_\epsilon(I_{K_2})\cap\tau_a(I_{K_1})=\emptyset\}.
\eeq
Because of the monotonicity of the set $A^{K_1,K_2}_{L,\epsilon}$ with respect to $\epsilon$,
\beqn
A_L^{K_1,K_2}\defas A^{K_1,K_2}_{L,\ve_0(L)} = \bigcap_{\epsilon\leq \ve_0(L)} A^{K_1,K_2}_{L,\epsilon}.
\eeqn

For large enough $L$ and any $\epsilon\leq\ve_0(L)$, $B_\epsilon(I_{K_2})$ is the union of $N$
disjoint intervals almost centered at the points $c_1,\dots,c_N$, which
do not depend on the parameter $a$. Moreover, $\tau_a(I_{K_1})$ is the union of at most
$N$ intervals in~$\bS$ which can be obtained from
$\tau_0(I_{K_1})$ by a rigid rotation. 
Moreover, using \eqref{eq:I_K}, we obtain
\beqn
\m(B_\epsilon(I_K))\le CKL^{-1}+2\epsilon
\qquad
\text{and}
\qquad
\m(\tau_a(I_K))\le CK^2L^{-1},
\eeqn
if $K/L$ is small enough. These observations together yield
\beq\label{eq:A_L_bound}
\begin{split}
\m(A^{K_1,K_2}_L)&\ge
1-N\m(\tau_a(I_{K_1}))-N\m(B_{\ve_0(L)}(I_{K_2}))
\\
&\ge1-CK_1^2L^{-1}-CK_2L^{-1}-2\ve_0(L).
\end{split}
\eeq
Also note that the complement of $A^{K_1,K_2}_L$ is the union of at most $N^2$ intervals, meaning that the same is true of $A^{K_1,K_2}_L$ itself.

Since the conditions of Theorem~\ref{thm:erg} are assumed, we can use the formula in~\eqref{eq:Lyap} for ergodic measures $\mu$ to bound the Lyapunov exponent $\lambda_a(L)$:
\begin{align*}
 \lambda_a(L)
&\ge \int_{(I_{K_2})^c}\log |\tau_a'|\,d\mu+\int_{I_1}\log
|\tau_a'|\,d\mu\\
&\ge (1-\mu(I_{K_2}))\log
K_2+\sup_{I_1}\rho\cdot\int_{I_1}\log|\tau_a'|\,d\m
\\
&\ge \biggl(1-\sup_{I_{K_2}}\rho\cdot CK_2 L^{-1}\biggr)\log
K_2-\sup_{I_1}\rho\cdot CL^{-1},
\end{align*}
where on the last line Lemma~\ref{lem:log_int} and \eqref{eq:I_K} have been used.

To resume the above estimate, we use the upper bound \eqref{eq:bound2} on the invariant density $\rho$ on $I_{K_2}\supset I_1$. Let $x_0\in I_{K_2}$. 
Note that $\tau_a^{-1}B_\ve(x_0)$ consists of finitely many disjoint
intervals, and thus so does its complementary set. We label all these intervals
of $\bS$ clockwise by $J_1,\cdots, J_{2k},J_{2k+1}=J_1$ so that, for any $i\in\{1,\dots,k\}$,
$$J_{2i-1}\subset \tau_a^{-1}B_\ve(x_0) \quad\text{while}\quad J_{2i}\subset
\left(\tau_a^{-1}B_\ve(x_0)\right)^c. 
$$
For $a\in A^{K_1,K_2}_L$, we have $$J_{2i-1}\cap
I_{K_1}=\emptyset,\quad \ie,\quad \big|\tau_a'|_{J_{2i-1}}\big|\ge K_1,$$
such that
$$\m(J_{2i-1})\le 2\ve/K_1.
$$
For any point $z\in\bS$,
the interval $B_\ve(z)$ can overlap with no more than
$$
M=2\ve\biggl(\min\Bigl\{\m(J_{i})+\m(J_{i+1}) :    1\leq i\leq 2k\; \& \; \tau_a'(x)\neq 0\, \forall x\in J_i\cup J_{i+1} \Bigr\}\biggr)^{-1}+1+N
$$
of the intervals $J_{2i-1}$. Here $N$ is the number of those intervals $J_{2i}$ which contain a critical point of $\tau_a$. On the other hand, if $\tau_a'\neq 0$ on $J_{2i}$, $\tau_a$ maps the interval $J_{2i}$ onto $(B_\ve(x_0))^c$. In this case,
the bound $|\tau_a'(x)|\le CL$ implies $\m(J_{2i})\ge C^{-1}L^{-1}(1-2\ve)$. As $C^{-1}L^{-1}\leq \ve\leq \frac 13$ holds for large $L$, 
$$M\leq  CL\ve$$ uniformly in $z$ and $L$, for such $L$. Therefore,  \eqref{eq:bound2} shows that
$$
\sup_{I_{K_2}}\rho\le \frac{CL\ve}{4\ve^2} \m(J_{2i-1}) 
\le
C\frac{L}{K_1},
$$
which in combination with the earlier bound on $\lambda_a(L)$ results in
\beq\label{eq:Lyap_bound}
\begin{split}
 \lambda_a(L)
& \ge \biggl(1-C\frac{K_2}{K_1}  \biggr)\log
K_2-C\frac1{K_1}
\\
& \ge \biggl(1-C\frac{K_2}{K_1}  \biggr)\log
K_2.
\end{split}
\eeq

Finally, define
\beqn
\ve_0(L)=L^{-1/2},\quad K_1 = (L/\log L)^{1/2}, \quad\text{and}\quad K_2=L^{1/2}/\log L.
\eeqn
Then the parameter set
\beqn
A_L = A_L^{K_1,K_2}
\eeqn
has all desired properties, as can be checked using \eqref{eq:A_L_bound} and \eqref{eq:Lyap_bound}, so that the theorem has been verified in the special case in which $\ve\leq \ve_0(L)$ holds for all large $L$. Now, assume that $\ve>\ve_0(L)$ for an unbounded set of values of $L$ and observe that
\beqn
\begin{split}
& \liminf_{L\to\infty} \inf_{a\in A_L}\frac{\lambda_a(L)}{\log L}
= \min \! \left(\liminf_{\substack{L\to\infty \\ L\,:\,\ve\le \ve_0(L)}} \inf_{a\in A_L}\frac{\lambda_a(L)}{\log L}, \,\liminf_{\substack{L\to\infty \\ L\,:\,\ve> \ve_0(L)}}\inf_{a\in A_L}\frac{\lambda_a(L)}{\log L}\right).
\end{split}
\eeqn
The theorem follows by combining the previous special case with Theorem~\ref{thm:large_smear}.
\end{proof}




\end{document}